\documentclass[12pt]{amsart}
\usepackage{amsmath,amssymb,latexsym,amsthm,newlfont,enumerate}

\usepackage{graphicx}
\usepackage[all]{xy}

\theoremstyle{plain}

\newtheorem{thm}{Theorem}[section]

\newtheorem{lem}[thm]{Lemma}

\newtheorem{cor}[thm]{Corollary}
\newtheorem{con}[thm]{Conjecture}

\theoremstyle{definition}

\newtheorem{dfn}[thm]{Definition}

\newtheorem{rem}[thm]{Remark}

\theoremstyle{remark}

\DeclareMathOperator{\mult}{mult}

\DeclareMathOperator{\im}{Im}
\DeclareMathOperator{\Int}{int}
\DeclareMathOperator{\relint}{relint}
\DeclareMathOperator{\Supp}{Supp}
\DeclareMathOperator{\Bs}{Bs}

\DeclareMathOperator{\Div}{Div}

\DeclareMathOperator{\Amp}{Amp}
\DeclareMathOperator{\Nef}{Nef}

\DeclareMathOperator{\NEbar}{\overline{NE}}

\DeclareMathOperator{\ddiv}{div}
\DeclareMathOperator{\codim}{codim}

\DeclareMathOperator{\Proj}{Proj}

\newcommand{\R}{\mathbb{R}}

\newcommand{\kb}{\mathbf{k}}
\newcommand{\Q}{\mathbb{Q}}
\newcommand{\N}{\mathbb{N}}
\newcommand{\LL}{\mathbb{L}}
\newcommand{\Z}{\mathbb{Z}}

\newcommand{\D}{\mathbf{D}}
\newcommand{\C}{\mathbb{C}}

\newcommand{\B}{\mathbf{B}}

\newcommand{\n}{\mathbf{n}}

\newcommand{\OO}{\mathcal{O}}

\newcommand{\eff}{\mathrm{eff}}

\newcommand{\mcal}{\mathcal}
\newcommand{\mfrak}{\mathfrak}
\newcommand{\aord}{o}
\newcommand{\ab}{\mathbf{a}}
\newcommand{\vb}{\mathbf{v}}
\newcommand{\ub}{\mathbf{u}}
\newcommand{\wb}{\mathbf{w}}

\title{New outlook on the Minimal Model Program, II}
\date{\today}

\author{Alessio Corti} \address{Department of Mathematics, Imperial
  College London, 180 Queen's Gate, London SW7 2AZ, UK}
\email{a.corti@imperial.ac.uk}

\author{Vladimir Lazi\'c} \address{Mathematisches Institut, Universit\"at Bayreuth, 95440 Bayreuth, Germany}
\email{vladimir.lazic@uni-bayreuth.de}

\keywords{Birational geometry, Minimal Model Program, finite generation}

\subjclass[2010]{14E30}

\begin{document}

\begin{abstract}
  We prove that the finite generation of adjoint rings proved in
  Cascini and Lazi\'c~\cite{CL10} implies all the foundational results
  of the Minimal Model Program: the Rationality, Cone and Contraction
  theorems, the existence of flips, and termination of flips with
  scaling in the presence of a big boundary.
\end{abstract}

\maketitle
\bibliographystyle{amsalpha}

\tableofcontents

\section{Introduction}
\label{sec:introduction}

In this paper we prove that the following results are simple
consequences of the finite generation of the canonical ring: the
Rationality, Cone and Contraction theorems, the existence of flips,
and termination of flips with scaling in the presence of a big
boundary. We regard these statements as the \emph{foundational
  theorems of the Minimal Model Program}, in the sense that they allow
us to run the Minimal Model Program to the extent that is known from
\cite{BCHM}.

Our results are especially relevant when combined with the recent
proof by Cascini and Lazi\'c \cite{CL10} of the finite generation of
the canonical ring, by means of a self-contained and efficient
argument that depends only on the Kawamata--Viehweg vanishing theorem
and makes no use of the foundational theorems just mentioned, or the
Minimal Model Program.\footnote{This was done first in \cite{Laz09}, surveyed
in \cite{corti11:_finit_lazic}. The paper \cite{CL10} is an evolution
and replacement of \cite{Laz09} using several of the original ideas
and constructions.}

Thus \cite{CL10} and the present paper taken together constitute a
self-contained presentation with complete proofs of all the results
that one needs to run the Minimal Model Program as in \cite{BCHM}. We
believe that this is the most efficient and simplest available
treatment of the material. 

The Minimal Model Program started out with the insight that we should
study curves rather than divisors. In this paper, we almost always
work with divisors, not curves. For example, in our treatment of the
Rationality, Cone and Contraction theorems, we prove a statement
concerning the dual cone of nef divisors, which is equivalent to those
three theorems combined.

By {\em finite generation of the canonical ring\/} we
mean the statement of Theorem~\ref{thm:3} below.  Assuming it, we
prove the Rationality, Cone and Contraction theorem \cite[Theorem~3.5,
Theorem~3.7]{KM98}, the existence of flips \cite[Theorem~A]{HM10}, and
termination of flips with scaling in the presence of a big boundary
\cite[Corollary~1.4.2]{BCHM}.

The fact that finite generation of the canonical ring implies (a
subset of) the foundational theorems is not new: for instance, it is
implicit in \cite{HK00}, and it may have been known earlier. At that
time, it would have been hard to imagine that a direct way to prove
the finite generation of the canonical ring was possible. Indeed, the
Minimal Model Program was set up initially, at least in part, as a way
to attack finite generation. Although we claim no originality when it
comes to the statements, we believe that our proofs are particularly
transparent, simple and quick.

The key tool in this paper is the following statement \cite[Corollary
1.1.9]{BCHM}, proved in \cite{CL10} by a self-contained argument based
on induction on the dimension and vanishing theorems.

\begin{thm}[{\cite[Theorem~A]{CL10}}]
  \label{thm:3}
  Let $X$ be a smooth projective variety and $A$ an ample
  $\mathbb{Q}$-divisor on $X$.  Let $\Delta_i$ be $\Q$-divisors on $X$
  such that $\lfloor \Delta_i\rfloor=0$ for $i=1,\dots,r$, and such
  that the support of $\sum\Delta_i$ has simple normal crossings.
  Then the adjoint ring
$$
R(X; K_X+\Delta_1+A, \dots, K_X+\Delta_r+A)
$$
is finitely generated.
\end{thm}

The definition of adjoint ring is given in Section~\ref{sec:notation-conventions}.

A majority of experts would agree that the most important open
problems left in the field are the unconditional termination of flips
with scaling (i.e.\ termination with scaling without assuming that the
boundary is big), and the Abundance conjecture. We prove in
Theorems~\ref{thm:scaling_unconditional} and~\ref{thm:1} that both of
these are consequences of Conjecture \ref{con:fg}, which is a version
of finite generation of the canonical ring that is more general than
Theorem~\ref{thm:3}. We note that it is easy to prove that,
conversely, unconditional termination with scaling and the Abundance
conjecture imply Conjecture \ref{con:fg}, see \cite[Lemma 2.6]{HX11}.

\begin{con}\label{con:fg}
  Let $(X,\Delta)$ be a projective klt pair such that $K_X+\Delta$ is
  pseudoeffective, and let $A$ be an ample $\Q$-divisor on $X$. Then
  the adjoint ring
  \[
  R(X;K_X+\Delta,K_X+\Delta+A)
  \]
  is finitely generated.
\end{con}

We hope that this stronger version of finite generation can be
attacked directly, perhaps using some of the ideas and constructions
from \cite{CL10}. We remark that it is {\em a priori\/} a weaker
version of the following well-known conjecture.  

\begin{con}
  \label{con:FGfull}
  Let $X$ be a normal projective variety, and let $\Delta_i$ be $\Q$-divisors
  on $X$ such that the pairs $(X,\Delta_i)$ are klt for $i=1,\dots,r$. Then the adjoint ring
\[
R(X; K_X+\Delta_1, \dots, K_X+\Delta_r)
\]
  is finitely generated.
\end{con}

It is a pleasure to acknowledge our intellectual debt to
V.~V.~Shokurov and Y.-T.~Siu. Shokurov wrote two fundamental papers on
the existence of flips \cite{Sho92, Sho03}, in which he reduced the
problem of construction of flips for a general klt pair to that of
\emph{pl flips}, and then constructed \mbox{3-fold} pl flips by
proving that a certain ring on a smooth surface is finitely
generated. Both of these key steps are present ``in spirit'' in
\cite{CL10}. Using their lifting lemma---which goes back
to~\cite{Siu98}---Hacon and M\textsuperscript{c}Kernan later realised
that Shokurov's ring is an adjoint ring. The work of Shokurov and Siu
provided the impetus for us to turn head over heels the traditional
perspective where one was supposed to establish the Minimal Model
Program first, and then prove the finite generation of the canonical
ring as a consequence.  \vspace{5mm}

\paragraph{\bf Acknowledgements}
We would like to thank P.~Cascini, F.~Catanese, S.~Coughlan, O.~Fujino,
A.-S.~Kaloghiros, K.~Matsuki and M.~Musta{\c{t}}{\u{a}} for many useful
conversations and comments. We are grateful to the referee for many useful suggestions.

\section{Notation and conventions}
\label{sec:notation-conventions}

  In this paper all algebraic varieties and schemes are defined over
  $\C$. We use $\Q_+$ and $\R_+$ to denote the sets of nonnegative
  rational and real numbers.

  Let $X$ be a normal projective variety and $\kb\in \{\Z,\Q,\R\}$. We
  denote by $\Div_\kb(X)$ the group of
  $\kb$-Cartier divisors on $X$, and by $N^1(X)_\kb$ and
  $N_1(X)_\kb$ the groups of $\kb$-Cartier divisors and $1$-cycles on
  $X$ with coefficients in $\kb$ modulo numerical equivalence. The ample and
  nef cones in $N^1(X)_\R$ are denoted by $\Amp (X)$ and $\Nef (X)$.
  Further, $\NEbar(X)$ denotes the closed cone of curves in $N_1(X)_\R$.

  Many arguments in this paper take place in $\Div_\R (X)$ or a finite
  dimensional subspace thereof not up to equivalence---it is crucial
  for us to distinguish this space from $N^1 (X)_\R$.

  We say that a $\kb$-divisor $D$ is \emph{$\kb$-effective} if there
  exists a divisor $D'\geq0$ such that $D\sim_\kb D'$. If $D$ is
  $\kb$-effective, then it is pseudoeffective. We denote by
  $\Div^{\eff}_\kb(X)$ the set of $\kb$-effective divisors in
  $\Div_\kb (X)$.  The {\em stable base locus} of a $\Q$-effective
  $\Q$-divisor $D$ is $\B(D)=\bigcap\Bs|mD|$, where the intersection
  is over all integers $m>0$ that are sufficiently divisible.

  A {\em pair} $(X,\Delta)$ consists of a normal projective variety
  $X$ and a $\Q$-divisor $\Delta\geq0$ on $X$ such that $K_X+\Delta$
  is $\Q$-Cartier. When $(X,\Delta)$ is a pair, $K_X+\Delta$ is called
  an {\em adjoint divisor}. In particular, this includes the canonical
  class $K_X$ when $\Delta=0$.

  If $X$ is a normal projective variety with field of fractions
  $k(X)$, and $D$ an $\R$-divisor on $X$, then $\OO_X(D)\subset k(X)$
  is the sheaf given by
\[
\OO_X(D)(U)=\bigl\{f\in k(X)\mid \ddiv_U f+D_{|U} \geq 0  \bigr\}
\]
  for every Zariski open set $U\subseteq X$.  We denote by $H^0(X,D)$ the
  group of global sections of this sheaf.

  In this note we only use divisorial rings of the form
\[
R=R(X;D_1, \dots, D_r)=\bigoplus_{(n_1,\dots, n_r)\in \N^r} H^0(X,
n_1D_1+\dots + n_rD_r),
\]
where $D_1,\dots, D_r$ are $\Q$-Cartier $\Q$-divisors on $X$ (not
necessarily $\Q$-effective). If all $D_i$ are adjoint divisors, then
$R(X;D_1, \dots, D_r)$ is an {\em adjoint ring}. The {\em support\/}
of $R$ is the cone
\[
\Supp R=\bigg( \sum_{i=1}^r\R_+D_i\bigg) \cap\Div^{\eff}_\R(X)
\subset \Div_\R(X).
\]
The choice of divisors $D_1, \dots, D_r$ gives the {\em
  tautological\/} linear map
\[
\textstyle \D \colon \R^r \ni (\lambda_1, \dots, \lambda_r)\mapsto
\sum \lambda_iD_i \in \Div_\R(X).
\]
  If $\LL\subseteq\Z^r$ is a finite index subgroup, then a ring of the form
\[
R(X;\D_{|\N^r \cap \LL})=\bigoplus_{\n\in \N^r \cap \LL}
H^0\bigl(X,\D(\n)\bigr)
\]
is called a \emph{Veronese subring of finite index\/} of $R(X;D_1,
\dots, D_r)$. 

A {\em geometric valuation\/} $\Gamma$ on a normal variety $X$ is a
valuation on $k(X)$ given by the order of vanishing at the generic
point of a prime divisor on some birational model $f\colon
Y\to X$. If $D$ is an $\R$-Cartier divisor on $X$, by abusing
notation we use $\mult_\Gamma D$ to denote $\mult_\Gamma f^*D$.

If $\vb, \wb$ are vectors in a real vector space, we denote by
\[
[\vb, \wb]=\bigl\{t\wb+(1-t)\vb \mid t \in [0,1]\bigr\}
\]
  the closed line segment between $\vb$ and $\wb$. We similarly write
  $(\vb,\wb)$ and $(\vb, \wb]$ for the open and half-open segments.

\section{Simple consequences of finite generation}
\label{sec:two-coroll-finite}

The following useful lemma is well known.

\begin{lem}\label{lem:2}
  Let $X$ be a normal projective variety and let $D_1,\dots,D_r$ be
  $\Q$-Cartier $\Q$-divisors on $X$. The ring $R=R(X;D_1,\dots,D_r)$
  is finitely generated if and only if any of its Veronese subrings of
  finite index is finitely generated. In particular, if $D_i\sim_\Q
  D_i'$ and if $R$ is finitely generated, then the ring
  $R'=R(X;D_1',\dots,D_r')$ is finitely generated.
\end{lem}

\begin{proof}
  Let $\LL\subseteq\Z^r$ be a subgroup of index $d$ giving the
  Veronese subring $R_\LL$ of $R$. Then for any $f\in R$ we have
  $f^d\in R_\LL$, so $R$ is an integral extension of $R_\LL$.
  Furthermore, one can write an action of the group $\Z^r/\LL$ on $R$
  such that $R_\LL$ is the ring of invariants. Now the first claim
  follows from theorems of Emmy Noether on finiteness of integral
  closure and of ring of invariants, and the second claim follows by
  noting that $R$ and $R'$ have isomorphic Veronese subrings of finite
  index.
\end{proof}

We use the following small variation of Theorem~\ref{thm:3}:

\begin{thm}
\label{thmA}
  Let $X$ be a normal projective variety, and let $\Delta_i$ be $\Q$-divisors
  on $X$ such that each pair $(X,\Delta_i)$ is klt for $i=1,\dots,r$.
  \begin{enumerate}
\item If $A$ is an ample $\Q$-divisor on $X$, then the adjoint ring
\[
R(X; K_X+\Delta_1+A, \dots, K_X+\Delta_r+A)
\]
  is finitely generated.
\item If $\Delta_i$ are big, then the adjoint ring
\[
R(X;K_X+\Delta_1,\dots,K_X+\Delta_r)
\]
is finitely generated.
\end{enumerate}
\end{thm}

\begin{proof}
  To prove (1), let $f\colon Y\to X$ be a log resolution of the pair $(X,\sum\Delta_i)$. For each $i$, let $\Gamma_i,G_i\geq0$
  be $\Q$-divisors on $Y$ without common components such that $G_i$ is $f$-exceptional and $K_Y+\Gamma_i=f^*(K_X+\Delta_i)+G_i$.
  Let $F\geq0$ be an $f$-exceptional $\Q$-divisor on $Y$ such that $A'=f^*A-F$ is ample, and such that $\lfloor\Gamma_i'\rfloor=0$
  for all $i$, where $\Gamma_i'=\Gamma_i+F$. Then
\[
R(X; K_X+\Delta_1+A, \dots, K_X+\Delta_r+A)\simeq R(Y; K_Y+\Gamma_1'+A', \dots, K_X+\Gamma_r'+A'),
\]
  and (1) follows from Theorem~\ref{thm:3}.

  For (2), let $H\geq0$ be an ample $\Q$-divisor on $X$ such that there exist
  divisors $E_i\geq0$ with $\Delta_i\sim_\Q E_i+H$. For a rational
  $0<\varepsilon\ll1$ set $A=\varepsilon H$ and
  $\Delta_i'=(1-\varepsilon)\Delta_i+\varepsilon E_i$. Then
  $K_X+\Delta_i\sim_\Q K_X+\Delta_i'+A$, and the pair $(X,\Delta_i'+A)$
  is klt for every $i$ since $(X,\Delta_i)$ is klt and $\varepsilon\ll1$. The ring
\[
R(X;K_X+\Delta_1^\prime+A,\dots,K_X+\Delta_r^\prime+A)
\]
  is finitely generated by (1), and (2) follows from Lemma~\ref{lem:2}.
\end{proof}

\begin{dfn}
  \label{dfn:1}
  Let $X$ be a normal projective variety, $D\in\Div^{\eff}_\R(X)$, and let
  $\Gamma$ be a geometric valuation of $X$. The
  {\em asymptotic order of vanishing\/} of $D$ along $\Gamma$ is
\[\aord_\Gamma (D)=\inf\{\mult_\Gamma D'\mid D\sim_\R D'\geq0\}\in \R.\]
\end{dfn}

If $X$ is a normal projective variety, $D$ a Cartier divisor on $X$
with $|D|\neq\emptyset$, and $\Gamma$ a geometric valuation of $X$, we write
\[\mult_\Gamma |D|=\min_{D^\prime \in |D|} \mult_\Gamma D^\prime \in \Z.\]
We use the following result without explicit mention.

\begin{lem}\label{lem:3}
  Let $X$ be a normal projective variety, $D\in\Div^{\eff}_\Q(X)$, and
  let $\Gamma$ be a geometric valuation of $X$.  Then, for all
  sufficiently divisible positive integers $p$, $o_\Gamma(D)=\inf
  \{\frac1p\mult_\Gamma|pD|\}$.
\end{lem}

\begin{proof}
  Let $F\geq0$ be an $\R$-divisor such that $F\sim_\R D$. Then there
  exist real numbers $r_1,\dots,r_k$ and rational functions
  $f_1,\dots,f_k\in k(X)$ such that $F=D+\sum_{i=1}^k r_i \ddiv f_i$.
  Let $W\subset \Div_{\mathbb R}(X)$ be the finite dimensional
  subspace based by the components of $D$ and all $\ddiv f_i$, let
  $\|\cdot\|$ be the sup norm on $W$, and let $W_0\subset W$ be the
  subspace of divisors $\mathbb R$-linearly equivalent to zero. Note
  that $W_0$ is a rational subspace of $W$, and consider the quotient
  map $\pi\colon W\longrightarrow W/W_0$. Then the set $\mathcal
  G=\{G\in \pi^{-1}(\pi(D))\mid G\ge 0\}$ is nonempty as it contains
  $F$, and it is cut out in $W$ by rational hyperplanes.  Thus, for
  every $\varepsilon>0$, $\mathcal G$ contains a $\mathbb Q$-divisor
  $D'\ge 0$ such that $D\sim_{\mathbb Q}D'$ and
  $\|F-D'\|<\varepsilon$. This proves the lemma.
\end{proof}

\begin{thm}
  \label{thm:2}
  Let $X$ be a normal projective variety and let $D_1, \dots, D_r$ be
  $\Q$-Cartier $\Q$-divisors on $X$.  Assume that the ring $R=R(X;D_1,
  \dots, D_r)$ is finitely generated, and let $\D \colon \R^r \ni
  (\lambda_1, \dots, \lambda_r)\mapsto \sum \lambda_iD_i \in
  \Div_\R(X)$ be the tautological map.
  \begin{enumerate}
  \item The support of $R$ is a rational polyhedral cone.
  \item Suppose that $\Supp R$ contains a big divisor. If
    $D\in \sum \R_+ D_i$ is pseudoeffective, then $D\in\Supp R$.
  \item There is a finite rational polyhedral subdivision $\Supp
    R=\bigcup \mcal{C}_i$ such that $\aord_\Gamma$ is a linear
    function on $\mcal{C}_i$ for every geometric valuation $\Gamma$ of
    $X$. Furthermore, there is a \emph{coarsest} subdivision with this
    property, in the sense that, if $i$ and $j$ are distinct, there is
    at least one geometric valuation $\Gamma$ of $X$ such that (the
    linear extensions of) $\aord_\Gamma|_{\mcal C_i}$ and
    $\aord_\Gamma|_{\mcal C_j}$ are different.
  \item There is a finite index subgroup $\LL \subseteq \Z^r$ such that
    for all $\n\in \N^r \cap \LL$, if $\D(\n)\in\Supp R$, then
   \[
     \aord_\Gamma \big(\D(\n)\big)=\mult_\Gamma |\D(\n)|
   \]
    for all geometric valuations $\Gamma$ of $X$.
  \end{enumerate}
\end{thm}

\begin{proof}
  Claim (1) is obvious. For (2), let $B$ be a big divisor in $\Supp
  R$. Observe that every divisor in the interval $(D,B]$ is big, hence
  $(D,B]\subseteq \Supp R$. But then $[D,B]\subseteq \Supp R$ since
  $\Supp R$ is closed by (1).

  We extract the proofs of (3) and (4) verbatim from the proof of
  \cite[Theorem 4.1]{ELMNP}.  Consider the system of ideals $(\mfrak
  b_\n)_{\n\in\N^r}$, where $\mfrak b_\n$ is the base ideal of the
  linear system $|\D(\n)|$. This is a finitely generated system, so by
  \cite[Proposition 4.7]{ELMNP} there is a rational polyhedral
  subdivision $\R_+^r=\bigcup\mcal D_i$ and a positive integer $d$
  such that for every $i$, if $e_1^i,\dots,e_s^i$ are generators of
  $\N^r\cap\mcal D_i$, then
\[
\textstyle\overline{\mfrak b_{d\sum_j p_je_j^i}}=\overline{\prod_j \mfrak b_{d e_j^i}^{p_j}}
\]
  for every $(p_1,\dots,p_s)\in\N^s$. Since a valuation of an ideal is equal to that of
  its integral closure, we deduce that for every geometric valuation $\Gamma$ of $X$,
  $o_\Gamma$ is linear on each of the cones $\mcal C_i=\Supp R\cap \D(\mcal D_i)$, and we can take $\LL=(d\Z)^r$.
  The existence of the coarsest subdivision as in (3) follows directly from
  convexity of asymptotic order functions.
\end{proof}

The following statement forms part of \cite[Theorem~B]{CL10}; here we
prove it as an easy consequence of Theorem~\ref{thmA}.

\begin{cor}
  \label{cor:1}
  Let $(X,\Delta)$ be a projective klt pair where $\Delta$ is
  big. If $K_X+\Delta$ is pseudoeffective, then it is $\Q$-effective.
\end{cor}

\begin{proof}
  Let $A$ be an ample $\Q$-divisor on $X$ such that the pair
  $(X,\Delta+A)$ is klt. By Theorem~\ref{thmA}, the adjoint ring
  \[
  R=R(X;K_X+\Delta,K_X+\Delta+A)
  \]
  is finitely generated, and $\Supp R$ contains the big divisor
  $K_X+\Delta+A$ by construction. The conclusion now follows from
  Theorem~\ref{thm:2}(2).
\end{proof}

\begin{lem}\label{lem:ords}
  Let $X$ be a normal projective variety and $D\in\Div_\Q^{\eff}(X)$.
  \begin{enumerate}
  \item If $D$ is semiample, then $\aord_\Gamma (D)=0$ for every
    geometric valuation $\Gamma$ of $X$.
  \item Assume that $R(X,D)$ is finitely generated. If $\aord_\Gamma
    (D)=0$ for every geometric valuation $\Gamma$ of $X$, then $D$
    is semiample.
  \end{enumerate}
\end{lem}

\begin{proof}
  Assume $D$ is semiample. Then a positive integer multiple $pD$ is
  basepoint free, thus clearly all $\aord_{\Gamma} (D)=0$.

  Conversely, if $R(X,D)$ is finitely generated and
  $\aord_\Gamma (D)=0$ for a valuation $\Gamma$, then
  the centre of $\Gamma$ on $X$ is not in $\B(D)$ by Theorem~\ref{thm:2}(4).
  Since every point on $X$ is the centre of some valuation $\Gamma$, we have $\B (D)=\emptyset$ and thus $D$ is semiample.
\end{proof}

Next we derive a special case of Kawamata's Basepoint free theorem as a consequence of Theorem \ref{thm:3}.

\begin{cor}
  \label{cor:2}
  Let $(X,\Delta)$ be a projective klt pair where $\Delta$ is big.
  If $K_X+\Delta$ is nef, then it is semiample.
\end{cor}

\begin{proof}
  Let $A$ be an ample $\Q$-divisor on $X$ such that the pair
  $(X,\Delta+A)$ is klt. By Theorem~\ref{thmA}, the ring
\[
R=R(X; K_X+\Delta, K_X+\Delta+A)
\]
is finitely generated, and $\Supp
R=\R_+(K_X+\Delta)+\R_+(K_X+\Delta+A)$ by Corollary~\ref{cor:1}.  For
each $\varepsilon>0$, the divisor $K_X+\Delta+\varepsilon A$ is ample,
thus $o_\Gamma(K_X+\Delta+\varepsilon A)=0$ for every geometric
valuation $\Gamma$ of $X$. Therefore, all $o_\Gamma$ are identically
zero on $\Supp R$ by Theorem~\ref{thm:2}(3), and thus $K_X+\Delta$ is
semiample by Lemma \ref{lem:ords}(2).
\end{proof}

We will use the following corollary in several arguments below.

\begin{cor}
  \label{cor:5}
  Let $X$ be a normal projective variety and let $D_1, \dots, D_r$ be
  $\Q$-Cartier $\Q$-divisors on $X$.  Assume that the ring $R=R(X;D_1,
  \dots, D_r)$ is finitely generated, and let $\Supp R=\bigcup
  \mcal{C}_i$ be a finite rational polyhedral subdivision such that,
  for every geometric valuation $\Gamma$ of $X$, $\aord_\Gamma$ is
  linear on $\mcal{C}_i$, as in Theorem~\ref{thm:2}(3). Denote by
  $\varphi \colon \sum \R_+ D_i\to N^1(X)_\R$ the natural projection,
  and assume that there exists $k$ such that $\mcal{C}_k \cap
  \varphi^{-1} \bigl(\Amp X\bigr)\neq \emptyset$.

  Then $\mcal{C}_k\subseteq\Supp R\cap \varphi^{-1}\bigl(\Nef
  X\bigr)$. If, additionally, the subdivision is coarsest, then
  $\mcal{C}_k=\Supp R\cap \varphi^{-1}\bigl(\Nef X\bigr)$.
\end{cor}

\begin{proof}
Note that by Theorem~\ref{thm:2}(3) all asymptotic
order functions $\aord_\Gamma$ are identically zero on $\mcal{C}_k$,
because they are so on the nonempty subset $\mcal{C}_k\cap
\varphi^{-1}\bigl(\Amp X\bigr)$; therefore, by Theorem~\ref{thmA}
and Lemma~\ref{lem:ords}(2), every element of $\mcal{C}_k$ is
semiample and thus $\mcal{C}_k\subseteq \varphi^{-1} \bigl(\Nef
X\bigr)$.

On the other hand, it is clear that all asymptotic order functions are
identically zero on $\Supp R\cap \varphi^{-1} \bigl(\Amp X\bigr)$.
Therefore, if we assume that the subdivision is coarsest, $\Supp R \cap
\varphi^{-1}\bigl(\Amp X\bigr)$ is entirely contained in $\mcal{C}_k$,
and then we conclude that also $\Supp R \cap \varphi^{-1}\bigl(\Nef X \bigr) \subseteq
\mcal{C}_k$ since $\mcal{C}_k$ is closed.
\end{proof}

\section{Rationality, Cone and Contraction theorem}

\begin{dfn} 
  \label{dfn:2}
  Let $W$ be a finite dimensional real vector space, $\mcal C\subset
  W$ a closed convex cone which spans $W$, and $\vb \in W$. The
  \emph{visible boundary} of $\mcal{C}$ from $\vb$ is the set
\[
V=\big\{\wb\in\partial\mcal C \mid [\vb,\wb]\cap\mcal C =\{\wb\}\big\} .
\]
Note that $V=\emptyset$ when $\vb$ is in the interior of $\mathcal C$,
and $V=\{\vb\}$ when $\vb\in\partial\mathcal C$.
\end{dfn}
\begin{figure}[htb]
\begin{center}
\includegraphics[width=0.46\textwidth]{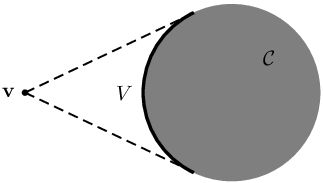}
\end{center}
\end{figure}

The following statement is Kawamata's reformulation of the Rationality, Cone
and Contraction theorem \cite{Kaw09}. Here, $\relint$
denotes the relative interior (i.e.\ the interior in the relative topology).

\begin{thm}[Rationality, Cone and Contraction theorem]\label{thm:cone}
  Let $(X,\Delta)$ be a projective klt pair. Let $V$ be the visible
  boundary of $\Nef(X)$ from the class
  $\vb_0\in N^1 (X)_\R$ of the divisor $K_X+\Delta$. Then:
  \begin{enumerate}
  \item every compact subset $F \subseteq \relint V$ is contained in a
  union of finitely many supporting rational hyperplanes;
\item every $\Q$-Cartier $\Q$-divisor on $X$ with class in $\relint V$
  is semiample.
  \end{enumerate}
\end{thm}

\begin{rem}
  \label{rem:2}
  Let $L$ be a $\Q$-Cartier $\Q$-divisor with class in $\relint V$. By
  the theorem, $L$ is semiample; in other words, there is a
  contraction morphism
\[
f_L \colon X \to Y
\]
  and $L=f_L^\ast A$ for an ample $\Q$-divisor $A$ on $Y$.

  If $F_L \subseteq V$ is the smallest face containing $L$, then we have
  natural identifications
\[
\Nef (Y) = F_L \subset N^1(Y)_\R=\langle F_L \rangle \subseteq N^1(X)_\R,
\]
  where $\langle F_L \rangle$ denotes the vector subspace generated by
  $F_L$.

  If $L$ and $M$ are $\Q$-divisors in $\relint V$, and $F_{L} \subseteq
  F_{M}$, then there is a factorization $f_L=g \circ f_M$.
  In particular, $M=f_M^\ast D$, where $D=g^\ast A$ is a semiample
  divisor.

  The $f_L$ are contractions of faces of $\NEbar(X)$ in
  Minimal Model Program; extremal contractions correspond to those $L$ that lie in
  the relative interiors of faces of maximal dimension.
\end{rem}

\begin{proof}[Proof of Theorem~\ref{thm:cone}]
  We work with $N^1(X)_\R$ equipped with the sup norm.

  We first prove (1). If $\vb_0\in \Nef (X)$, there is nothing to
  prove; thus, we can and will assume that $\vb_0\notin \Nef (X)$.
  Consider the cone $\mcal{C}=\R_+ \vb_0 + \Nef (X) $; by compactness
  of $F$, there is a rational number $0<\varepsilon\ll1$ and finitely
  many rational points $\ub_1,\dots,\ub_p\in\Int\mcal C$ such that
  $F\subseteq\Int\big(\bigcup
  B(\ub_i,\varepsilon)\big)\subseteq\Int\mcal C$, where $B(\ub_i,
  \varepsilon)$ denotes the closed ball. Since we are working in the
  sup norm, $B(\ub_i, \varepsilon)$ are cubes, thus the convex hull
  $\mcal B$ of $\bigcup B(\ub_i,\varepsilon)$ is a rational polytope:
  denote by $\wb_1,\dots,\wb_m$ its vertices. Then
\[
\wb_j\in\Int\mcal{C}= \bigcup_{\ab \in \Amp (X)} (\vb_0,\ab),
\]
so there exist rational ample classes $\ab_j$ and rational numbers
$t_j\in(0,1)$ such that $\wb_j=t_j\vb_0+(1-t_j)\ab_j$. For each $j$,
choose an ample $\Q$-divisor $A_j$ with class $\frac{1-t_j}{t_j}\ab_j$
such that the pair $(X, \Delta+A_j)$ is klt; then $\wb_j$ is the class
of the divisor $t_j(K_X+\Delta+A_j)$.  By Theorem~\ref{thmA}(2), the
adjoint ring
\[
R=R(X;K_X+\Delta+A_1, \dots, K_X+\Delta+A_m)
\]
is finitely generated, and denote by $\varphi \colon \sum \R_+
(K_X+\Delta+A_j)\to N^1(X)_\R$ the natural projection; by
construction, $F \subseteq \varphi \bigl(\sum \R_+
(K_X+\Delta+A_j)\bigr)=\R_+\mathcal B$. Then $\mcal{L}=\Supp
R\subset \Div_\R(X)$ is a rational polyhedral cone by
Theorem~\ref{thm:2}(1), and since $F\subseteq\Int\mcal
B\cap\partial\Nef(X)$, we have $\mcal B\cap\Amp(X)\neq\emptyset$, so
$\mcal L$ contains ample divisors as $\varphi(\mathcal
L)\subseteq\R_+\mathcal B$.  Therefore $\varphi^{-1}(F) \subseteq
\mcal{L}$ by Theorem~\ref{thm:2}(2), as every class in $F$ is
pseudoeffective.

Let $\mcal{L}=\bigcup \mcal{L}_k$ be the coarsest subdivision as in
Theorem~\ref{thm:2}(3). Then there exists $k$ such that $\mcal{L}_k
\cap \varphi^{-1} \bigl(\Amp X\bigr)\neq \emptyset$, and thus
$\mcal{L}_k=\mcal{L}\cap \varphi^{-1}\bigl(\Nef X\bigr)$ by
Corollary~\ref{cor:5}. Since $\varphi^{-1} (F) \subseteq \mcal{L} \cap
\varphi^{-1} \bigl(\Nef X\bigr)$, this implies $\varphi^{-1}(F)
\subseteq \partial \mcal{L}_k$, hence we have (1).

We now prove (2). Let $D$ be a $\Q$-divisor on $X$ with class $\vb \in
\relint V$.  As above, there is a rational ample class $\ab$ and a
rational number $t\in (0,1)$ such that $\vb=t\vb_0+(1-t)\ab$.  Then we
can choose an ample $\Q$-divisor $A$ with class
$\bigl(\frac{1-t}{t}\bigr)\ab$ such that the pair $(X,\Delta+A)$ is
klt and
\[
D\sim_\Q t(K_X+\Delta +A).
\]
Now $D$ is semiample by Corollary~\ref{cor:2}.
\end{proof}


\section{Birational contractions}

A birational map $f\colon X\dashrightarrow Y$ between normal varieties
is a {\em birational contraction\/} if the rational map $f^{-1}$ does
not contract divisors; in other words, some open subset of $X$ is
isomorphic to an open set $U\subseteq Y$ with $\codim_Y (Y\setminus
U)\geq2$. If additionally $X$ and $Y$ are $\Q$-factorial, and
$(p,q)\colon W\longrightarrow X\times Y$ is a resolution of $f$, then
we define the map $f^*\colon \Div_\R(Y)\longrightarrow\Div_\R(X)$ as
$f^*=p_*\circ q^*$; this does not depend on the choice of $W$.  Note
that $f^*=f^{-1}_*$ if $f$ is an isomorphism in codimension $1$.
Extremal contractions and flips are examples of birational
contractions.

\begin{lem}
  \label{rem:5}
  Let $X$ and $Y$ be $\Q$-factorial projective varieties,
  let $f \colon X\dashrightarrow Y$ be a birational contraction,
  and let $\tilde f\colon k(X)\simeq k(Y)$ be the induced isomorphism. Then:
  \begin{enumerate}
  \item $f_*\ddiv_X\varphi=\ddiv_Y\tilde f(\varphi)$ for every
    $\varphi\in k(X)$;
  \item for every geometric valuation $\Gamma$ on $k(X)$ and for every
    $\varphi\in k(X)$ we have $\mult_\Gamma(\ddiv_X
    \varphi)=\mult_\Gamma\big(\ddiv_Y\tilde f(\varphi)\big)$;
\item if $f$ is an isomorphism in codimension one, then
  $f_*\colon\Div_\R(X)\longrightarrow\Div_\R(Y)$ is an isomorphism, and
  for every $D\in\Div_\R(X)$ the map $\tilde f$ restricts to the
  isomorphism $H^0(X,D)\simeq H^0(Y,f_*D)$.
  \end{enumerate}
\end{lem}
\begin{proof}
  For (1), let $U\subseteq X$ and $V\subseteq Y$ be open subsets such
  that $f_{|U}\colon U\longrightarrow V$ is an isomorphism and
  $\codim_Y(Y\setminus V)\geq2$. Then obviously
  $(f_*\ddiv_X\varphi)_{|V}=(\ddiv_Y\tilde f(\varphi))_{|V}$, thus the
  claim.  The second claim is easily verified on a common resolution
  of $X$ and $Y$.  If additionally $\codim_X(X\setminus U)\geq2$, then
  (3) follows from $H^0(U,D)\simeq H^0(V,f_*D)$ since $X$ and $Y$ are
  normal.
\end{proof}

The following lemma will be used in the proof of termination with scaling.

\begin{lem}\label{lem:1}
  Let $X$ and $Y$ be $\Q$-factorial projective varieties and let $f
  \colon X\dashrightarrow Y$ be a birational map which is an
  isomorphism in codimension one. Let $\mcal
  C\subset\Div_\R^{\eff}(X)$ be a cone, and fix a geometric valuation
  $\Gamma$ of $X$.  Then the asymptotic order of vanishing
  $\aord_\Gamma$ is linear on $\mcal C$ if and only if it is linear on
  $f_*\mcal C\subset \Div_\R^{\eff} (Y)$.
\end{lem}

\begin{proof}
  For every $D\in\mcal C$, write $V_D=\{D_X-D\mid D\sim_\R D_X\textrm{
    and }D_X\geq0\}\subset\Div_\R(X)$ and $W_D=\{D_Y-f_*D\mid
  f_*D\sim_\R D_Y\text{ and }D_Y\geq0\}\subset\Div_\R(Y)$.  Note that
  the elements of $V_D$ and $W_D$ are $\R$-linear combinations of
  principal divisors.  By Lemma~\ref{rem:5} we have the isomorphism
  $f_*|_{V_D}\colon V_D\simeq W_D$, and $\mult_\Gamma P_X=\mult_\Gamma
  f_*P_X$ for every $P_X\in V_D$. Therefore
\[
o_\Gamma(D)-\mult_\Gamma D=\inf_{P_X\in V_D}\mult_\Gamma
P_X=\inf_{P_X\in V_D}\mult_\Gamma f_*P_X=o_\Gamma(f_*D)-\mult_\Gamma
f_*D,
\]
hence the function $o_\Gamma(\cdot)-o_\Gamma\big(f_*(\cdot)\big)\colon
\mcal C\longrightarrow\R$ is equal to the linear map
$\mult_\Gamma(\cdot)-\mult_\Gamma f_*(\cdot)$.  The claim now follows.
\end{proof}

The following lemma is well known \cite[Lemma~1.7]{HK00}.

\begin{lem}\label{lem:HK}
  Let $X$, $Y$ and $Z$ be normal varieties projective over a variety
  $S$, and $f\colon X\dashrightarrow Y$ and $g\colon X\dashrightarrow
  Z$ birational contractions.  Suppose that there exist a divisor $A$
  on $Y$ which is ample over $S$ and a divisor $B$ on $Z$ which is nef
  over $S$ such that
\[
f^*A+F=g^*B+G,
\]
  where $F\geq0$ is $f$-exceptional and $G\geq0$ is $g$-exceptional. Then the
  birational map $f\circ g^{-1}\colon Z\dasharrow Y$ is a morphism. \qed
\end{lem}
\begin{proof}
  By passing to a common resolution, we can assume that $f$ and $g$
  are morphisms. Then by the Negativity Lemma we have $E=F$ and
  $f^*A=g^*B$, so $f$ contracts all the curves that are contracted by
  $g$. Now the result follows from the Rigidity Lemma \cite[Lemma
  1.15]{Deb01}.
\end{proof}

The following is an easy consequence of the property of separatedness
of schemes; see for instance \cite[Ex.~II.4.2]{Har77}.

\begin{lem}\label{lem:iso}
  Let $X$ be a reduced scheme, $Y$ a separated scheme, and let
  $f$ and $g$ be two morphisms from $X$ to $Y$.  Assume that
  $f_{|U}=g_{|U}$ on a Zariski dense open subset $U\subseteq X$. Then
  $f=g$.\qed
\end{lem}

\section{Termination with scaling}

\begin{dfn}
  Let $(X,\Delta)$ be a projective klt pair and $A$ a big $\Q$-divisor
  on $X$ such that $K_X+\Delta+A$ is nef. The {\em nef threshold} of
  $(X,\Delta)$ with respect to $A$ is
\[
\lambda=\lambda(X,\Delta,A)=\inf \{t\in\R_+ \mid K_X+\Delta+tA \text{ is
  nef}\,\}.
\]
\end{dfn}

The following lemma is a central ingredient in the Minimal Model
Program with scaling, to be discussed shortly.

\begin{lem}\label{lem:nullflip}
  Let $(X,\Delta)$ be a $\Q$-factorial projective klt pair such that
  $K_X+\Delta$ is not nef, and let $A$ be a big $\Q$-divisor on $X$
  such that $(X,\Delta+A)$ is klt and $K_X+\Delta+A$ is nef. Let
  $\lambda=\lambda(X,\Delta,A)$ be the nef threshold. Then
  $\lambda\in\Q_+$, and there exists an extremal ray $R \subset
  \NEbar (X)$ with $(K_X+\Delta+\lambda A)\cdot R=0$ and
  $(K_X+\Delta)\cdot R<0$.
\end{lem}

\begin{proof}
  Denote by $\varphi\colon\Div_\R(X)\longrightarrow N^1(X)_\R$ the
  natural projection, and let $\|\cdot\|$ be any norm on $N^1(X)_\R$.
  Pick finitely many big $\Q$-divisors $\Delta_1, \dots, \Delta_r$
  such that:
  \begin{enumerate}
  \item $\|\varphi(\Delta+\lambda A)-\varphi(\Delta_i)\|\ll1$ for all $i$;
  \item writing $\mcal C=\sum_{i=1}^r \R_+(K_X+\Delta_i) \subset \Div_\R(X)$, we have
      $K_X+\Delta+\lambda A\in\Int\mcal C$, and the dimension of the cone
      $\varphi(\mcal C) \subset N^1(X)_\R$ is $\dim N^1 (X)_\R$.
  \end{enumerate}
 As in the proof of Theorem~\ref{thmA}(2), there exist an ample $\R$-divisor $H$ and a $\Q$-divisor $B\geq0$ such that $\Delta+\lambda A\sim_\R H+B$ and the pair $(X,B)$ is klt. Write $H_i=\Delta_i-B$ for every $i$. Then $H_i$ are ample $\Q$-divisors since $\|\varphi(H)-\varphi(H_i)\|\ll1$ by (1), hence there are ample $\Q$-divisors $A'_i\sim_\Q H_i$ such that all pairs $(X,A_i'+B)$ are klt, and obviously $\Delta_i\sim_\Q A'_i+B$.

  Therefore, by Theorem~\ref{thmA}(2) the ring
\[
R=R(X;K_X+\Delta_1, \dots, K_X+\Delta_r)
\]
is finitely generated. By Corollary~\ref{cor:5} the cone $\Supp R \cap
\varphi^{-1}\big(\Nef (X) \big)$ is rational polyhedral, and there is
a rational codimension one face $F\ni \varphi(K_X+\Delta+\lambda A)$
of $\Nef (X)$. This implies $\lambda\in\Q_+$, and we choose
$R\subset\NEbar(X)$ to be the extremal ray dual to $F$.
\end{proof}

\paragraph{\textbf{The Minimal Model Program with scaling}}

Let $(X,\Delta)$ be a $\Q$-factorial projective klt pair, and $A$ a
big $\Q$-divisor on $X$. Assume that $(X,\Delta+A)$ is klt and
$K_X+\Delta+A$ is nef.  The Minimal Model Program with scaling by $A$
is the following version of the Minimal Model Program for
$K_X+\Delta$. Starting with $X_1=X$, $\Delta_1=\Delta$ and $A_1=A$, we
define an inductive sequence of rational maps
\[
(X_1,\Delta_1,\lambda_1A_1) \overset{f_1} \dasharrow \cdots
\overset{f_{i-1}} \dasharrow (X_i, \Delta_i,\lambda_iA_i)
\overset{f_i} \dasharrow (X_{i+1}, \Delta_{i+1},\lambda_{i+1}A_{i+1})
\overset{f_{i+1}} \dasharrow \cdots ,
\]
where $\Delta_i$, $A_i$ are the proper transforms of $\Delta$, $A$ on
$X_i$, $\lambda_i=\lambda (X_i,\Delta_i,A_i)$ is the nef threshold,
and $f_i\colon X_i \dasharrow X_{i+1}$ is the extremal contraction or
flip corresponding to a $(K_{X_i}+\Delta_i)$-extremal ray $R_i$
with $(K_{X_i}+\Delta_i+\lambda_i A_i)\cdot R_i=0$ as in
Lemma~\ref{lem:nullflip}. Note that
$K_{X_i}+\Delta_i+\lambda_{i-1}A_i$ is nef by Remark \ref{rem:2}, thus
the sequence $\lambda_i$ is nonincreasing.

First we prove the existence of flips \cite[Theorem~A]{HM10}.

\begin{thm}
  Let $(X,\Delta)$ be a $\Q$-factorial projective klt pair such that
  $K_X+\Delta$ is not nef, and $R\subset \NEbar (X)$ an extremal ray
  such that $(K_X+\Delta)\cdot R<0$.  Let $f_R\colon X \to Y$ be the
  contraction morphism of $R$ as in Remark~\ref{rem:2}, and assume
  $f_R$ is small.  Then the flip of $f_R$ exists.
\end{thm}

\begin{proof}
  It is well-known \cite[Corollary~6.4(2)]{KM98} that the existence of
  the flip is equivalent to the finite generation of the relative
  adjoint ring
\[
R(X/Y,K_X+\Delta)=\bigoplus_{n\geq 0}(f_R)_* \OO_X \bigl(n(K_X+\Delta)\bigr),
\]
in which case $\Proj_Y R(X/Y,K_X+\Delta)$ is the flip.
We may assume that $Y$ is affine, and the theorem
is equivalent to the finite generation of the canonical ring
$R(X,K_X+\Delta)$. The flipping contraction is birational so we can
apply Theorem~\ref{thmA} in the relative over $Y$ case.
\end{proof}

\begin{rem}\label{rem:6}
  Let $(X,\Delta)$ be a projective klt pair, and $f\colon
  X\dashrightarrow Y$ a composite of $(K_X+\Delta$)-divisorial
  contractions and $(K_X+\Delta)$-flips. Then by \cite[Lemma
  3.38]{KM98}, for every resolution $(p,q)\colon W\longrightarrow X\times
  Y$ of $f$ we have
\[
p^*(K_X+\Delta)=q^*(K_Y+f_*\Delta)+E,
\]
where $E> 0$ is a $q$-exceptional divisor.  Therefore $f$ cannot be an
isomorphism, and the above formula implies
\[
H^0(X,K_X+\Delta)\simeq H^0(Y,K_Y+f_*\Delta).
\]
\end{rem}

Now we can establish termination of flips with scaling with big
boundary \cite[Corollary~1.4.2]{BCHM}.

\begin{thm}\label{thm:scalingbig}
  Let $(X_1,\Delta_1)$ be a projective $\Q$-factorial klt pair with
  $\Delta_1$ big.  Let $A_1$ be a big $\Q$-divisor on $X$ such that
  $(X_1,\Delta_1+A_1)$ is klt and $K_{X_1}+\Delta_1+A_1$ is nef, and
  let $\lambda_1=\lambda(X_1,\Delta_1,A_1)$.  Then any sequence
\[
(X_1,\Delta_1,\lambda_1A_1) \overset{f_1} \dasharrow \cdots
\overset{f_{i-1}} \dasharrow (X_i, \Delta_i,\lambda_iA_i)
\overset{f_i} \dasharrow (X_{i+1}, \Delta_{i+1},\lambda_{i+1}A_{i+1})
\overset{f_{i+1}} \dasharrow \cdots
\]
 of flips of the Minimal Model Program with scaling of $A_1$ terminates.
\end{thm}

\begin{proof}
  Suppose by contradiction that an infinite sequence of such flips
  exists; then in particular $\lambda_i>0$ for all $i$.  Denote by
  $\varphi\colon\Div_\R(X_1)\longrightarrow N^1(X_1)_\R$ the natural
  projection and let $\|\cdot\|$ be any norm on $N^1(X_1)_\R$.  Choose
  big $\Q$-divisors $H_1, \dots, H_r$ on $X_1$ such that:
  \begin{enumerate}
  \item $\|\varphi(\Delta_1+\lambda_1 A_1)-\varphi(H_j)\|\ll1$ for all $j$;
  \item writing $\mcal
    C^1=\R_+(K_{X_1}+\Delta_1)+\sum_{j=1}^r\R_+(K_{X_1}+H_j)\subset\Div_\R(X_1)$,
    we have $K_{X_1}+\Delta_1+\lambda_1A_1\in\Int\mcal C^1$, and the
    dimension of the cone $\varphi(\mcal C^1)\subset N^1(X_1)_\R$ is
    $\dim N^1(X_1)_\R$.
\end{enumerate}
For each $i$, let $H_j^i$ be the proper transforms of $H_j$ on $X_i$, and write
\[
R_i=R(X_i;K_{X_i}+\Delta_i, K_{X_i}+H_1^i, \dots ,K_{X_i}+H_r^i).
\]
By Lemma~\ref{rem:5}(3) we have $R_i\simeq R_1$ for all $i$, and these
rings are finitely generated as in the beginning of the proof of
Lemma~\ref{lem:nullflip}.

By construction, the cone $\mcal C^1$ contains an open neighbourhood
of the nef divisor $K_{X_1}+\Delta_1+\lambda_1 A_1$, so it contains
ample divisors in its interior, and thus the cone $\varphi(\Supp R_1) \subset
N^1(X_1)_\R$ also has dimension $\dim N^1(X_1)_\R$.  Let $\Supp R_1 =
\bigcup\mcal C_k^1$ be the finite rational polyhedral subdivision as in
Theorem~\ref{thm:2}(3).

Denote by $\mcal{C}_k^i\subset \Div_\R (X_i)$ the proper transform of
$\mcal{C}_k^1$ and by $\mcal{C}^i\subset \Div_\R (X_i)$ the proper
transform of $\mcal{C}^1$. By Lemma~\ref{lem:1}, for every geometric
valuation $\Gamma$ the asymptotic order function $\aord_\Gamma$ is
linear on each $\mcal{C}_k^i$.

By construction, if $0<\lambda\leq \lambda_1$, then $K_{X_1}+\Delta_1 +
\lambda A_1\in \Int\mcal{C}^1$, so $K_{X_i}+\Delta_i+\lambda_iA_i\in\Int
\mcal{C}^i$ for every $i$. Since $K_{X_i}+\Delta_i+\lambda_iA_i$ is
nef, we have $K_{X_i}+\Delta_i+\lambda_iA_i\in \Supp R_i$ by Corollary
\ref{cor:1}. Hence, by Corollary~\ref{cor:5}, for
each $i$ there exists an index $k$ such that the image of
$\mcal{C}_k^i$ in $N^1(X_i)_\R$ is a subset of $\Nef(X_i)$. Therefore
\[
\varphi(\mcal{C}_k^1) \subseteq (f_{i-1}\circ\dots\circ f_1)^*\Nef(X_i).
\]
Since there are finitely many cones $\mcal C_k^1$, there are two
indices $p$ and $q$ such that the cones $(f_{p-1}\circ\dots\circ
f_1)^*\Nef(X_p)$ and $(f_{q-1}\circ\dots\circ f_1)^*\Nef(X_q)$ share a
common interior point. Thus, by Lemma \ref{lem:HK} the map
$X_p\dashrightarrow X_q$ is a morphism. But then, by Lemma
\ref{lem:iso}, it is an isomorphism, which contradicts
Remark~\ref{rem:6}.
\end{proof}

\begin{cor}\label{cor:4}
  Let $(X,\Delta)$ be a projective $\Q$-factorial klt pair where
  $\Delta$ is big, and let $A$ be an ample $\Q$-divisor on $X$ such
  that $(X,\Delta+A)$ is klt and $K_X+\Delta+A$ is nef. Then:
\begin{enumerate}
\item if $K_X+\Delta$ is pseudoeffective, the Minimal Model Program
  with scaling of $A$ terminates with a minimal model, and the
  canonical model of $(X,\Delta)$ exists;
\item if $K_X+\Delta$ is not pseudoeffective, the Minimal Model
  Program with scaling of $A$ terminates with a Mori fibre space.\qed
\end{enumerate}
\end{cor}

\begin{cor}\label{cor:3}
  Let $(X,\Delta)$ be a projective $\Q$-factorial klt pair such that
  $K_X+\Delta$ is not pseudoeffective, and $A$ an ample $\Q$-divisor
  on $X$ such that $(X,\Delta+A)$ is klt and $K_X+\Delta+A$ is nef.
  Then the Minimal Model Program with scaling of $A$ terminates with a
  Mori fibre space.
\end{cor}
\begin{proof}
  There exists $0<\mu\ll1$ such that $K_X+\Delta+\mu A$ is also not
  pseudoeffective, thus all $(K_X+\Delta)$-extremal contractions are
  $(K_X+\Delta+\mu A)$-extremal contractions.  We conclude by
  Corollary~\ref{cor:4}.
\end{proof}

Finally we prove that Conjecture~\ref{con:fg} implies unconditional
Minimal Model Program with scaling.

\begin{thm}\label{thm:scaling_unconditional}
  Let $(X,\Delta)$ be a projective $\Q$-factorial klt pair, and let
  $A$ be an ample $\Q$-divisor on $X$ such that $(X,\Delta+A)$ is klt
  and $K_X+\Delta+A$ is nef.  If Conjecture~\ref{con:fg} holds, then
  the Minimal Model Program with scaling of $A$ terminates.
\end{thm}

\begin{proof}
  By Corollary \ref{cor:3} we can assume that $K_X+\Delta$ is
  pseudoeffective. Consider the Minimal Model Program with scaling of
  $A_1=A$ starting from the pair $(X_1,\Delta_1)=(X,\Delta)$:
\[
(X_1,\Delta_1,\lambda_1A_1) \overset{f_1} \dasharrow \cdots
\overset{f_{i-1}} \dasharrow (X_i, \Delta_i,\lambda_iA_i)
\overset{f_i} \dasharrow (X_{i+1}, \Delta_{i+1},\lambda_{i+1}A_{i+1})
\overset{f_{i+1}} \dasharrow \cdots ,
\]
and consider the adjoint rings
\[
R_i=R(X_i;K_{X_i}+\Delta_i, K_{X_i}+\Delta_i+\lambda_i A_i).
\]
  By Remark \ref{rem:6} we have
$R_i\simeq R(X_1;K_{X_1}+\Delta_1, K_{X_1}+\Delta_1+\lambda_i A_1)$,
  thus all $R_i$ are finitely generated by Conjecture~\ref{con:fg}. Note that here we invoke Conjecture~\ref{con:fg} on
  $X_1$ and not on $X_i$, since the divisor $A_i$ is, in general, only big and not ample.

  Assume that there exists an index $i_0$ and a sequence of flips $f_i$ for $i\geq i_0$. We will prove that this
  sequence is finite, and thus the program terminates.

  By Theorem~\ref{thm:2}(2) we have
 \[\Supp R_{i_0}=\R_+(K_{X_{i_0}}+\Delta_{i_0})+
  \R_+(K_{X_{i_0}}+\Delta_{i_0}+\lambda_{i_0} A_{i_0}).
\]
Let $\Supp R_{i_0}=\bigcup\mcal C_k^{i_0}$ be a rational polyhedral
subdivision as in Theorem~\ref{thm:2}(3), and let
$\mcal{C}_k^i\subset \Div_\R(X_i)$ denote the proper transform of
$\mcal{C}_k^{i_0}$ for $i\geq i_0$. By Lemma~\ref{lem:1},
for each geometric valuation $\Gamma$ on $k(X_i)$, the function $o_\Gamma$ is linear on
$\mcal C^i_k$.

  Assume that, for some index $k$,
  $K_{X_i}+\Delta_i+\lambda_iA_i\in\Int\mcal C^i_k$. By
  Corollary~\ref{cor:2}, $K_{X_i}+\Delta_i+\lambda_iA_i$ is semiample,
  hence $\aord_\Gamma ( K_{X_i}+\Delta_i+\lambda_iA_i )=0$ for every
  geometric valuation $\Gamma$ of $X_i$.  Since the functions
  $\aord_\Gamma$ are linear and nonnegative on $\mcal C^i_k$, they all must
  then be identically zero on $\mcal{C}^i_k$.  Then, by Lemma~\ref{lem:ords}(2) every divisor
  in $\mcal C_k^i$ is semiample, and thus nef, so it can not be
  true that $\lambda_i$ is \emph{smallest} such that
  $K_{X_i}+\Delta_i+\lambda_iA_i$ is nef, a contradiction.

  Thus, if $K_{X_i}+\Delta_i+\lambda_iA_i$ is in
  $\mcal{C}^i_k$, it is on one of the two boundary rays of
  $\mcal{C}^i_k$, and therefore $K_{X_{i_0}}+\Delta_{i_0}+\lambda_iA_{i_0}$
  is on one of the two boundary rays of
  $\mcal{C}_k^{i_0}$. Since there are finitely many cones $\mcal{C}_k^{i_0}$,
  the set $\{\lambda_i\}$ is finite. Write
  $\lambda=\min\{\lambda_i\}$. If $\lambda=0$, then $K_{X_i}+\Delta_i$
  is nef for $i\gg0$ and the program stops. If $\lambda>0$,
  choose $0<\mu<\lambda$.
  Then every $(K_{X_{i_0}}+\Delta_{i_0})$-flip is a $(K_{X_{i_0}}+\Delta_{i_0}+\mu A_{i_0})$-flip,
  and every sequence of $(K_{X_{i_0}}+\Delta_{i_0}+\mu A_{i_0})$-flips with scaling of $A_{i_0}$
  is finite by Theorem~\ref{thm:scalingbig}.
\end{proof}

\section{Conjecture \ref{con:fg} implies Abundance}
\label{sec:abundance}

\begin{thm}\label{thm:1}
  Conjecture \ref{con:fg} implies log abundance. In other words, if
  $(X,\Delta)$ is a projective klt pair and $K_X+\Delta$ is nef,
  then $K_X+\Delta$ is semiample.
\end{thm}

\begin{proof}
The proof is almost verbatim the proof of Corollary \ref{cor:2}. The
only difference is at the beginning of the proof when, instead of
invoking Theorem~\ref{thmA}, we invoke Conjecture~\ref{con:fg}.
\end{proof}

\section{Relative case}
\label{sec:relative}

In this section we briefly sketch how all the results above can be
extended to the case of a projective morphism of normal
quasi-projective varieties $\pi\colon X\longrightarrow S$. In general,
$\pi$ has a Stein factorization $\pi=p\circ \pi^\prime$, where
$\pi^\prime \colon X \longrightarrow S^\prime$ has connected fibres and $p\colon
S^\prime \longrightarrow S$ is finite. Since the Minimal Model Program for the
morphism $\pi$ is the same as the Minimal Model Program for the
morphism $\pi^\prime$, we assume that $\pi$ has connected fibres to
start with.

It is enough to prove relative versions of
Theorem~\ref{thmA}, Theorem~\ref{thm:2}, Lemma~\ref{lem:ords} and
Lemma~\ref{lem:1}, since all other results can be derived from these
by copying verbatim the proofs above in the absolute case.

Relative divisorial rings over $S$ are sheaves of $\OO_S$-algebras of
the form
\[
R=R(X/S;D_1, \dots, D_r)=\bigoplus_{(n_1,\dots, n_r)\in \N^r}
\pi_*\OO_X(\lfloor n_1D_1+\dots + n_rD_r\rfloor),
\]
where $D_1,\dots, D_r$ are $\Q$-Cartier $\Q$-divisors on $X$. The
support $\Supp R$ of this cone is the subcone of $\sum\R_+
D_i\subset\Div_\R(X)$ spanned by $\pi$-effective integral divisors.

First, we remark that the results in \cite{CL10} remain valid if
instead of projective varieties we consider varieties projective over
affine varieties. In particular, we have the following relative
version of Theorem~\ref{thmA}.

\begin{thm}
\label{thmB}
  Let $\pi\colon X\longrightarrow S$ be a projective morphism between normal varieties, and let $\Delta_i$ be $\Q$-divisors on $X$ such that the pairs $(X,\Delta_i)$ are klt for $i=1,\dots,r$.
  \begin{enumerate}
\item If $A$ is a $\pi$-ample $\Q$-divisor on $X$, then the relative adjoint ring
\[
R(X/S; K_X+\Delta_1+A, \dots, K_X+\Delta_r+A)
\]
  is finitely generated.
\item If $\Delta_i$ are $\pi$-big, then the relative adjoint ring
\[
R(X/S;K_X+\Delta_1,\dots,K_X+\Delta_r)
\]
is finitely generated.
\end{enumerate}
\end{thm}

If $D$ is a $\pi$-effective integral
divisor, as in \cite[9.1.17]{Laz04b} the natural homomorphism
$\pi^*\pi_*\OO_X(D)\longrightarrow \OO_X(D)$ determines the {\em base
  ideal of $D$ relative to $\pi$}:
$$\mathfrak b(\pi,|D|)=\im\big(\pi^*\pi_*\OO_X(D)\otimes\OO_X(-D)\longrightarrow \OO_X\big).$$
Let $\mu\colon Y\longrightarrow X$ be a log resolution of the ideal
$\mathfrak b(\pi,|D|)$, and let $F$ be an effective divisor on $Y$
such that $\mu^{-1}\mathfrak b(\pi,|D|)=\OO_Y(-F)$. For every
geometric valuation $\Gamma$ of $X$, we set $\mult_\Gamma\mathfrak
b (\pi,|D|)=\mult_\Gamma F$. Then the {\em asymptotic order of
  vanishing of $D$ along $\Gamma$ relative to $S$} is
$$o_{\Gamma/S}(D)=\liminf_{m\rightarrow\infty}\frac1m\mult_\Gamma
\mathfrak b (\pi,|mD|).$$ Another way to define this is by considering
relative linear systems.  As in \cite{BCHM}, for an $\R$-divisor $D$
and for $\mathbf k\in\{\Z,\Q,\R\}$, define the relative linear system
$|D/S|_\mathbf k=\{D'\mid D\sim_{\mathbf k,\pi}D'\geq0\}$. If $D$ is a
$\Q$-divisor we have $|D/S|_\R=|D/S|_\Q$ as in Lemma~\ref{lem:3}, and
if $D$ is integral, then $D$ is $\pi$-effective if and only if
$|D/S|\neq\emptyset$ by \cite[Lemma 3.2.1]{BCHM} and \cite[Lemma
8.3.4]{Kol07}. Furthermore, it follows from the projection formula
that $\mathfrak b(\pi,|D'|)=\mathfrak b(\pi,|D|)$ for $D'\in |D/S|$.
This proves that for every geometric valuation $\Gamma$ of $X$ and for
any $\R$-divisor $D$ on $X$, we can equivalently define the relative
asymptotic order of vanishing along $\Gamma$ as
$$o_{\Gamma/S}(D)=\inf\{\mult_\Gamma D'\mid D'\in|D/S|_\R\}.$$
One can then easily see that the relative analogues of Theorem~\ref{thm:2}, Lemma~\ref{lem:ords} and Lemma~\ref{lem:1} hold by making minor adaptations of the proofs of those results in the absolute case.

\bibliography{biblio}
\pagestyle{plain}
\end{document}